\documentclass{amsart}

\usepackage{amssymb,latexsym,amsfonts,amsmath}
\usepackage{multicol} 
\usepackage{graphicx}
\include{diagrams}
\usepackage{eso-pic}
\usepackage{epsfig}
\usepackage{graphicx}
\usepackage{color}
\usepackage{tikz}
\usetikzlibrary{automata}
\usetikzlibrary{shapes}

\usepackage{amssymb,latexsym,amsfonts,amsmath}
\usepackage{graphicx}

\newtheorem{theorem}{Theorem}[section]
\newtheorem{lemma}[theorem]{Lemma}

\newtheorem{definition}[theorem]{Definition}

\newtheorem{remark}[theorem]{Remark}
\numberwithin{equation}{section}

\newcommand{\R}{{\mathbb{R}}}

\newcommand{\N}{{\mathbb{N}}}

\begin{document}

\begin{abstract}
In this note, we propose coordinate-invariant notions of incremental Lyapunov function and provide characterizations of incremental stability in terms of existence of the proposed Lyapunov functions.
\end{abstract}

\title[Coordinate-invariant incremental Lyapunov functions]{Coordinate-invariant incremental Lyapunov functions}
\author[M. Zamani]{Majid Zamani$^1$} 
\author[R. Majumdar]{Rupak Majumdar$^2$} 

\address{$^1$Department of Electrical and Computer Engineering, Technische Universit\"at M\"unchen, D-80290 Munich, Germany.}
\email{zamani@tum.de}
\urladdr{http://www.hcs.ei.tum.de}

\address{$^2$Max Planck Institute for Software Systems, 67663 Kaiserslautern, Germany.}
\email{rupak@mpi-sws.org}
\urladdr{http://www.mpi-sws.org/$\sim$rupak}

\maketitle

\section{Control Systems and Stability Notions}
\subsection{Notation} 
The symbols $\mathbb{N}$, $\mathbb{R}$, $\mathbb{R}^+$ and $\mathbb{R}_0^+$ denote the set of natural, real, positive, and nonnegative real numbers, respectively. Given a vector $x\in\mathbb{R}^{n}$, we denote by $x_{i}$ the \mbox{$i$--th} element of $x$, and by $\Vert x\Vert$ the Euclidean norm of $x$; we recall that \mbox{$\Vert x\Vert=\sqrt{x_1^2+x_2^2+...+x_n^2}$}. Given a measurable function \mbox{$f:\mathbb{R}_{0}^{+}\rightarrow\mathbb{R}^n$}, the (essential) supremum of $f$ is denoted by $\Vert f\Vert_{\infty}$; we recall that \mbox{$\Vert f\Vert_{\infty}:=\text{(ess)sup}\{\Vert f(t)\Vert,t\geq0\}$}. Function $f$ is essentially bounded if $\Vert{f}\Vert_{\infty}<\infty$. For a given time $\tau\in\mathbb{R}^+$, define $f_{\tau}$ so that $f_{\tau}(t)=f(t)$, for any $t\in[0,\tau)$, and $f(t)=0$ elsewhere; $f$ is said to be locally essentially bounded if for any $\tau\in\mathbb{R}^+$, $f_{\tau}$ is essentially bounded. 
A continuous function \mbox{$\gamma:\mathbb{R}_{0}^{+}\rightarrow\mathbb{R}_{0}^{+}$}, is said to belong to class $\mathcal{K}$ if it is strictly increasing and \mbox{$\gamma(0)=0$}; $\gamma$ is said to belong to class $\mathcal{K}_{\infty}$ if \mbox{$\gamma\in\mathcal{K}$} and $\gamma(r)\rightarrow\infty$ as $r\rightarrow\infty$. A continuous function \mbox{$\beta:\mathbb{R}_{0}^{+}\times\mathbb{R}_{0}^{+}\rightarrow\mathbb{R}_{0}^{+}$} is said to belong to class $\mathcal{KL}$ if, for each fixed $s$, the map $\beta(r,s)$ belongs to class $\mathcal{K}_{\infty}$ with respect to $r$ and, for each fixed nonzero $r$, the map $\beta(r,s)$ is decreasing with respect to $s$ and $\beta(r,s)\rightarrow0$ as \mbox{$s\rightarrow\infty$}. 
A function \mbox{$\mathbf{d}:\R^n\times \R^n\rightarrow\mathbb{R}_{0}^{+}$} is a metric on $\R^n$ if for any $x,y,z\in\R^n$, the following three conditions are satisfied: i) $\mathbf{d}(x,y)=0$ if and only if $x=y$; ii) $\mathbf{d}(x,y)=\mathbf{d}(y,x)$; and iii) $\mathbf{d}(x,z)\leq\mathbf{d}(x,y)+\mathbf{d}(y,z)$. For a set $\mathcal{A}\subseteq\R^n$, and any $x\in\R^n$, $\mathbf{d}(x,\mathcal{A})$ denotes the point-to-set distance, defined by $\mathbf{d}(x,\mathcal{A})=\inf_{y\in\mathcal{A}}\mathbf{d}(x,y)$. 

\subsection{Control Systems\label{II.B}}

The class of control systems with which we deal in this note is formalized in
the following definition.
\begin{definition}
\label{Def_control_sys}A \textit{control system} is a quadruple:
\[
\Sigma=(\mathbb{R}^{n},\mathsf{U},\mathcal{U},f),
\]
where:
\begin{itemize}
\item $\mathbb{R}^{n}$ is the state space;
\item $\mathsf{U}\subseteq\mathbb{R}^{m}$ is the input set;
\item $\mathcal{U}$ is the set of all measurable functions of time from intervals of the form \mbox{$]a,b[\subseteq\mathbb{R}$} to $\mathsf{U}$ with $a<0$ and $b>0$; 
\item \mbox{$f:\mathbb{R}^{n}\times \mathsf{U}\rightarrow\mathbb{R}^{n}$} is a continuous map
satisfying the following Lipschitz assumption: for every compact set
\mbox{$Q\subset\mathbb{R}^{n}$}, there exists a constant $Z\in\mathbb{R}^+$ such that $\Vert
f(x,u)-f(y,u)\Vert\leq Z\Vert x-y\Vert$ for all $x,y\in Q$ and all $u\in \mathsf{U}$.
\end{itemize}
\end{definition}

A curve \mbox{$\xi:]a,b[\rightarrow\mathbb{R}^{n}$} is said to be a
\textit{trajectory} of $\Sigma$ if there exists $\upsilon\in\mathcal{U}$
satisfying:
\begin{equation}
\dot{\xi}(t)=f\left(\xi(t),\upsilon(t)\right),\label{eq0}
\end{equation}
for almost all $t\in$ $]a,b[$. 
We also write $\xi_{x\upsilon}(t)$ to denote the point reached at time $t$
under the input $\upsilon$ from initial condition $x=\xi_{x\upsilon}(0)$; this point is
uniquely determined, since the assumptions on $f$ ensure existence and
uniqueness of trajectories \cite{sontag1}. 
A control system $\Sigma$ is said to be forward complete if every trajectory is defined on an interval of the form $]a,\infty[$. We refer the interested readers to \cite{sontag} for sufficient and necessary conditions for a system to be forward complete. A control system $\Sigma$ is said to be smooth if $f$ is an infinitely differentiable function of its arguments.

\subsection{Stability notions}
Here, we recall the notions of incremental global asymptotic stability ($\delta_\exists$-GAS) and incremental input-to-state stability ($\delta_\exists$-ISS), presented in \cite{majid1}.

\begin{definition}[\cite{majid1}]
\label{delta_GAS1}
A control system $\Sigma$ is incrementally globally asymptotically stable ($\delta_\exists$-GAS) if it is forward complete and there exist a metric $\mathbf{d}$ and a $\mathcal{KL}$ function $\beta$ such that for any $t\in{\mathbb{R}_0^+}$, any $x,x'\in{\mathbb{R}^n}$ and any $\upsilon\in\mathcal{U}$ the following condition is satisfied:
\begin{equation}
\mathbf{d}\left(\xi_{x\upsilon}(t),\xi_{x'\upsilon}(t)\right) \leq\beta\left(\mathbf{d}\left(x,x'\right),t\right). \label{delta_GAS}%
\end{equation}
\end{definition}

As defined in~\cite{angeli}, $\delta$-GAS requires the metric $\mathbf{d}$ to be the Euclidean metric. However, Definition~\ref{delta_GAS1} only requires the existence of a metric. We note that while $\delta$-GAS is not generally invariant under changes of coordinates, $\delta_\exists$-GAS is. 

\begin{definition}[\cite{majid1}]
\label{dISS}
A control system $\Sigma$ is incrementally input-to-state stable ($\delta_\exists$-ISS) if it is forward complete and there exist a metric $\mathbf{d}$, a $\mathcal{KL}$ function $\beta$, and a $\mathcal{K}_{\infty}$ function $\gamma$ such that for any $t\in{\mathbb{R}_0^+}$, any $x,x'\in{\mathbb{R}^n}$, and any $\upsilon$, ${\upsilon}'\in\mathcal{U}$ the following condition is satisfied:
\begin{equation}
\mathbf{d}\left(\xi_{x\upsilon}(t),\xi_{x'{\upsilon}'}(t)\right) \leq\beta\left(\mathbf{d}\left(x,x'\right),t\right)+\gamma\left(\left\Vert{\upsilon}-{\upsilon}'\right\Vert_{\infty}\right). \label{delta_ISS}%
\end{equation}
\end{definition}

By observing (\ref{delta_GAS}) and (\ref{delta_ISS}), it is readily seen that $\delta_\exists$-ISS implies $\delta_\exists$-GAS while the converse is not true in general. Moreover, whenever the metric $\mathbf{d}$ is the Euclidean metric, $\delta_\exists$-ISS becomes $\delta$-ISS as defined in~\cite{angeli}. We note that while $\delta$-ISS is not generally invariant under changes of coordinates, $\delta_\exists$-ISS is.

Here, we introduce the following definition which was inspired by the notion of uniform global asymptotic stability with respect to sets in \cite{lin}.
\begin{definition}
A control system $\Sigma$ is uniformly globally asymptotically stable (U$_\exists$GAS) with respect to a set $\mathcal{A}$ if it is forward complete and there exist a metric $\mathbf{d}$, and a $\mathcal{KL}$ function $\beta$ such that for any $t\in\R_0^+$, any $x\in\R^n$ and any $\upsilon\in\mathcal{U}$ the following condition is satisfied:
\begin{equation}\label{UGAS}
\mathbf{d}(\xi_{x\upsilon}(t),\mathcal{A})\leq\beta(\mathbf{d}(x,\mathcal{A}),t).
\end{equation} 
\end{definition}
We discuss in the next section characterizations of $\delta_\exists$-GAS and $\delta_\exists$-ISS in terms of existence of incremental Lyapunov functions.

\subsection{Characterizations of incremental stability}
This section contains characterizations of $\delta_\exists$-GAS and $\delta_\exists$-ISS in terms of existence of incremental Lyapunov functions. We start by defining the new notions of $\delta_\exists$-GAS and $\delta_\exists$-ISS Lyapunov functions.

\begin{definition}
\label{delta_GAS_Lya}
Consider a control system $\Sigma=(\R^n,\mathsf{U},\mathcal{U},f)$ and a smooth function \mbox{$V:\mathbb{R}^n\times\mathbb{R}^n\rightarrow\mathbb{R}_0^+$}. Function $V$ is called a $\delta_\exists$-GAS Lyapunov function for $\Sigma$, if there exist a metric $\mathbf{d}$, $\mathcal{K}_{\infty}$ functions $\underline{\alpha}$, $\overline{\alpha}$, and $\kappa\in\mathbb{R}^+$ such that:
\begin{itemize}
\item[(i)] for any $x,x'\in\mathbb{R}^n$\\
$\underline{\alpha}(\mathbf{d}(x,x'))\leq{V}(x,x')\leq\overline{\alpha}(\mathbf{d}(x,x'))$;
\item[(ii)] for any $x,x'\in\mathbb{R}^n$ and any $u\in\mathsf{U}$\\
$\frac{\partial{V}}{\partial{x}}f(x,u)+\frac{\partial{V}}{\partial{x'}}f(x',u)\leq -\kappa V(x,x')$.
\end{itemize}
Function $V$ is called a $\delta_\exists$-ISS Lyapunov function for $\Sigma$, if there exist a metric $\mathbf{d}$, $\mathcal{K}_{\infty}$ functions $\underline{\alpha}$, $\overline{\alpha}$, $\sigma$, and $\kappa\in\mathbb{R}^+$ satisfying conditions (i) and:
\begin{itemize}
\item[(iii)] for any $x,x'\in\mathbb{R}^n$ and for any $u,u'\in\mathsf{U}$\\
\mbox{$\frac{\partial{V}}{\partial{x}}f(x,u)+\frac{\partial{V}}{\partial{x'}}f(x',u')\leq -\kappa V(x,x')+\sigma(\Vert{u}-u'\Vert)$}.
\end{itemize}
\end{definition}

\begin{remark}
Condition (iii) of Definition \ref{delta_GAS_Lya} can be replaced by: 
\begin{equation}
\frac{\partial{V}}{\partial{x}}f(x,u)+\frac{\partial{V}}{\partial{x'}}f(x',u')\leq -\rho(\mathbf{d}(x,x'))+\sigma(\Vert{u}-u'\Vert), \nonumber 
\end{equation}
where $\rho$ is a $\mathcal{K}_{\infty}$ function. It is known that there is no loss of generality in considering $ \rho(\mathbf{d}(x,x'))=\kappa V(x,y)$, by appropriately modifying the $\delta_\exists$-ISS Lyapunov function $V$ (see Lemma 11 in \cite{praly}).
\end{remark}

While $\delta$-GAS and $\delta$-ISS Lyapunov functions, as defined in \cite{angeli}, require the metric $\mathbf{d}$ in condition (i) in Definition \ref{delta_GAS_Lya} to be the Euclidean metric, Definition \ref{delta_GAS_Lya} only requires the existence of a metric. We note that while $\delta$-GAS and $\delta$-ISS Lyapunov functions are not invariant under changes of coordinates in general, $\delta_\exists$-GAS and $\delta_\exists$-ISS Lyapunov functions are. 

We now introduce the following definition which was inspired by the notion of uniform global asymptotic stability (UGAS) Lyapunov function in \cite{lin}.
\begin{definition}\label{UGAS_Lya}
Consider a control system $\Sigma$, a set $\mathcal{A}$, and a smooth function $V:\R^n\rightarrow\R_0^+$. Function $V$ is called a U$_\exists$GAS Lyapunov function, with respect to $\mathcal{A}$, for $\Sigma$, if there exist a metric $\mathbf{d}$, $\mathcal{K}_\infty$ functions $\underline\alpha$, $\overline\alpha$, and $\kappa\in\R^+$ such that:
\begin{itemize}
\item[(i)] for any $x\in\R^n$\\
$\underline{\alpha}(\mathbf{d}(x,\mathcal{A}))\leq{V}(x)\leq\overline{\alpha}(\mathbf{d}(x,\mathcal{A}))$;
\item[(ii)] for any $x\in\R^n$ and any $u\in\mathsf{U}$\\
$\frac{\partial{V}}{\partial{x}}f(x,u)\leq -\kappa V(x)$.
\end{itemize}
\end{definition}

The following theorem characterizes U$_\exists$GAS in terms of existence of a U$_\exists$GAS Lyapunov function.
\begin{theorem}\label{theorem1}
Consider a control system $\Sigma$ and a set $\mathcal{A}$. If $\mathsf{U}$ is compact and $\mathbf{d}$ is a metric such that the function $\psi(x)=\mathbf{d}(x,y)$ is continuous\footnote{Here, continuity is understood with respect to the Euclidean metric.} for any $y\in\R^n$ then the following statements are equivalent: 
\begin{itemize}
\item[(1)] $\Sigma$ is forward complete and there exists a U$_\exists$GAS Lyapunov function with respect to $\mathcal{A}$, equipped with the metric ${\mathbf{d}}$.
\item[(2)] $\Sigma$ is U$_\exists$GAS with respect to $\mathcal{A}$, equipped with the metric $\mathbf{d}$.
\end{itemize}
\end{theorem}

\begin{proof}
First we show that the function $\phi(x)=\mathbf{d}(x,\mathcal{A})$ is a continuous function with respect to the Euclidean metric. Assume $\{x_n\}_{n=1}^\infty$ is a converging sequence in $\R^n$ with respect to the Euclidean metric, implying: $x_n\rightarrow x^*$ as $n\rightarrow\infty$ for some $x^*\in\R^n$. By triangle inequality, we have:
\begin{equation}\label{ineq1}
\mathbf{d}\left(x^*,y\right)\leq\mathbf{d}\left(x^*,x_n\right)+\mathbf{d}\left(y,x_n\right),
\end{equation}
for any $n\in\N$ and $y\in\mathcal{A}$. Using inequality (\ref{ineq1}), we obtain:
\begin{eqnarray}
\label{ineq2}
\phi\left(x^*\right)=\inf_{y\in\mathcal{A}}\mathbf{d}\left(x^*,y\right)&\leq&\inf_{y\in\mathcal{A}}\left\{\mathbf{d}\left(x^*,x_n\right)+\mathbf{d}\left(y,x_n\right)\right\}\\\notag&=&\inf_{y\in\mathcal{A}}\mathbf{d}\left(y,x_n\right)+\mathbf{d}\left(x^*,x_n\right)\\\notag&=&\phi\left(x_n\right)+\mathbf{d}\left(x^*,x_n\right).
\end{eqnarray}
Using inequality (\ref{ineq2}) and the continuity assumption on $\mathbf{d}$, we obtain:
\begin{equation}\label{ineq3}
\phi\left(x^*\right)\leq\lim_{n\rightarrow\infty}\inf\phi\left(x_n\right),
\end{equation}
for any $n\in\N$, where limit inferior exists because of greatest lower bound property of real numbers \cite{radulescu}. By doing the same analysis, we have:
\begin{equation}\label{ineq4}
\phi\left(x^*\right)\geq\lim_{n\rightarrow\infty}\sup\phi\left(x_n\right),
\end{equation}
for any $n\in\N$. Using inequalities (\ref{ineq3}) and (\ref{ineq4}), we obtain:
\begin{equation}
\phi\left(x^*\right)=\lim_{n\rightarrow\infty}\phi\left(x_n\right),
\end{equation}
implying that $\phi$ is a continuous function. Since $\phi(x)=\mathbf{d}(x,\mathcal{A})$ is a continuous function, by choosing \mbox{$\omega_1(x)=\omega_2(x)=\mathbf{d}(x,\mathcal{A})$} and using Theorem 1 in \cite{teel}, the proof completes.
\end{proof}

Before showing the main results, we need the following technical lemma, inspired by Lemma 2.3 in \cite{angeli}.

\begin{lemma}\label{lemma1}
Consider a control system $\Sigma=(\R^n,\mathsf{U},\mathcal{U},f)$. If $\Sigma$ is $\delta_\exists$-GAS, then the control system $\widehat\Sigma=(\R^{2n},\mathsf{U},\mathcal{U},\widehat{f})$, where $\widehat{f}(\zeta,\upsilon)=\left[f(\xi_1,\upsilon)^T,f(\xi_2,\upsilon)^T\right]^T$, and $\zeta=\left[\xi_1^T,\xi_2^T\right]^T$, is U$_\exists$GAS with respect to the diagonal set $\Delta$, defined by:
\begin{equation}
\Delta=\left\{z\in\R^{2n}\\|\\\exists x\in\R^n:z=\left[x^T,x^T\right]^T\right\}.
\end{equation}
\end{lemma}

\begin{proof}
Since $\Sigma$ is $\delta_\exists$-GAS, there exists a metric $\mathbf{d}:\R^n\times\R^n\rightarrow\R_0^+$ such that property (\ref{delta_GAS}) is satisfied. Now we define a new metric $\widehat{\mathbf{d}}:\R^{2n}\times\R^{2n}\rightarrow\R^+_0$ by: 
\begin{equation}\label{metric}
\widehat{\mathbf{d}}(z,z')=\mathbf{d}(x_1,x'_1)+\mathbf{d}(x_2,x'_2),
\end{equation}
for any $z=\left[{x_1}^T,{x_2}^T\right]^T\in\R^{2n}$ and $z'=\left[{x'_1}^T,{x'_2}^T\right]^T\in\R^{2n}$. It can be readily checked that $\widehat{\mathbf{d}}$ satisfies all three conditions of a metric. Now we need to show that $\widehat{\mathbf{d}}(z,\Delta)$, for any $z=\left[x_1^T,x_2^T\right]^T\in\R^{2n}$, is proportional to $\mathbf{d}(x_1,x_2)$. We have:
\begin{eqnarray}
\label{ineq5}
\widehat{\mathbf{d}}(z,\Delta)&=&\inf_{z'\in\Delta}\widehat{\mathbf{d}}(z,z')=\inf_{x'\in\R^n}\widehat{\mathbf{d}}\left(\left[{\begin{array}{c}x_1\\x_2\\\end{array}}\right],\left[{\begin{array}{c}x'\\x'\\\end{array}}\right]\right)\\\notag&=&\inf_{x'\in\R^n}\left(\mathbf{d}(x_1,x')+\mathbf{d}(x_2,x')\right)\leq\mathbf{d}(x_1,x_2).
\end{eqnarray}
Since $\mathbf{d}$ is a metric, by using the triangle inequality, we have: $\mathbf{d}(x_1,x_2)\leq\mathbf{d}(x_1,x')+\mathbf{d}(x_2,x')$ for any $x'\in\R^n$, implying that $\mathbf{d}(x_1,x_2)\leq\widehat{\mathbf{d}}(z,\Delta)$. Hence, using (\ref{ineq5}), one obtains:
\begin{equation}\label{equality}
\mathbf{d}(x_1,x_2)\leq\widehat{\mathbf{d}}(z,\Delta)\leq\mathbf{d}(x_1,x_2)\Rightarrow\mathbf{d}(x_1,x_2)=\widehat{\mathbf{d}}(z,\Delta).
\end{equation} 
Using equality (\ref{equality}) and property (\ref{delta_GAS}), we have:
\begin{eqnarray}
\widehat{\mathbf{d}}\left(\zeta_{z\upsilon}(t),\Delta\right)&=&\mathbf{d}\left(\xi_{x_1\upsilon}(t),\xi_{x_2\upsilon}(t)\right)\\\notag&\leq&\beta\left(\mathbf{d}\left(x_1,x_2\right),t\right)=\beta\left(\widehat{\mathbf{d}}\left(z,\Delta\right),t\right),
\end{eqnarray}
for any $t\in\R_0^+$, and $\upsilon\in\mathcal{U}$, where $\zeta_{z\upsilon}=\left[\xi_{x_1\upsilon}^T,\xi_{x_2\upsilon}^T\right]^T$, and $z=\left[x_1^T,x_2^T\right]^T$. Hence, $\widehat\Sigma$ is U$_\exists$GAS with respect to $\Delta$.
\end{proof}

We can now state one of the main results, providing characterization of $\delta_\exists$-GAS in terms of existence of a $\delta_\exists$-GAS Lyapunov function.
\begin{theorem}
Consider a control system $\Sigma$. If $\mathsf{U}$ is compact and $\mathbf{d}$ is a metric such that the function $\psi(x)=\mathbf{d}(x,y)$ is continuous\footnote{Here, continuity is understood with respect to the Euclidean metic.} for any $y\in\R^n$ then the following statements are equivalent: 
\begin{itemize}
\item[(1)] $\Sigma$ is forward complete and there exists a $\delta_\exists$-GAS Lyapunov function, equipped with the metric ${\mathbf{d}}$.
\item[(2)] $\Sigma$ is $\delta_\exists$-GAS, equipped with the metric $\mathbf{d}$.
\end{itemize} 
\end{theorem}
\begin{proof}
The proof from (1) to (2) has been showed in Theorem 2.6 in \cite{majid4}, even in the absence of the compactness and continuity assumptions on $\mathsf{U}$ and $\mathbf{d}$, respectively. We now prove that (2) implies (1). Since $\Sigma$ is $\delta_\exists$-GAS, using Lemma \ref{lemma1}, we conclude that the control system $\widehat\Sigma$, defined in Lemma \ref{lemma1}, is U$_\exists$GAS with respect to the diagonal set $\Delta$. Since $\psi(x)=\mathbf{d}(x,y)$ is continuous for any $y\in\R^n$, it can be easily verified that $\widehat\psi(z)=\widehat{\mathbf{d}}(z,z')$ is also continuous for any $z'\in\R^{2n}$, where the metric $\widehat{\mathbf{d}}$ was defined in Lemma \ref{lemma1}. Using Theorem \ref{theorem1}, we conclude that there exists a U$_\exists$GAS Lyapunov function $V:\R^{2n}\rightarrow\R_0^+$, with respect to $\Delta$, for $\widehat\Sigma$.
Thanks to the special form of $\widehat\Sigma$ and using the equality (\ref{equality}), the function $V$ satisfies: 
\begin{itemize}
\item[(i)] $\underline{\alpha}(\mathbf{d}(x,x'))\leq{V}(x,x')\leq\overline{\alpha}(\mathbf{d}(x,x'))$;
\item[(ii)] $\frac{\partial{V}}{\partial{x}}f(x,u)+\frac{\partial{V}}{\partial{x'}}f(x',u)\leq -\kappa V(x,x')$,
\end{itemize}
for some $\mathcal{K}_\infty$ functions $\underline\alpha,\overline\alpha$ and some $\kappa\in\R^+$. Hence, V is a $\delta_\exists$-GAS Lyapunov function for $\Sigma$.
\end{proof}

Before providing characterization of $\delta_\exists$-ISS in terms of existence of a $\delta_\exists$-ISS Lyapunov function, we need the following technical lemma, inspired by Proposition 5.3 in \cite{angeli}. To state the following results, we need to define the function:
\begin{equation}\label{function}
\mathsf{sat}_\mathsf{U}(u)=\left\{
                \begin{array}{cc}
                 u&\mbox{if}~u\in\mathsf{U},\\ 
                 \arg\min_{u'\in\mathsf{U}}\left\Vert u'-u\right\Vert&\mbox{if}~u\notin\mathsf{U}.
                \end{array}
                \right.
\end{equation} 
As explained in \cite{angeli}, by assuming $\mathsf{U}$ is closed and convex and since $\Vert\cdot\Vert:\R^m\rightarrow\R^+_0$ is a proper, convex function, the definition (\ref{function}) is well-defined and the minimizer of $\left\Vert u'-u\right\Vert$ with $u'\in\mathsf{U}$ is unique. Moreover, by convexity of $\mathsf{U}$ we have:
\begin{equation}\label{convex}
\Vert\mathsf{sat}_\mathsf{U}(u_1)-\mathsf{sat}_\mathsf{U}(u_2)\Vert\leq\Vert u_1-u_2\Vert,~~~~\forall u_1,u_2\in\R^m.
\end{equation}

\begin{lemma}\label{lemma2}
Consider a control system $\Sigma=(\R^n,\mathsf{U},\mathcal{U},f)$, where $\mathsf{U}$ is closed and convex. If $\Sigma$ is $\delta_\exists$-ISS, equipped with a metric $\mathbf{d}$ such that $\psi(x)=\mathbf{d}(x,y)$ is continuous for any $y\in\R^n$, then there exists a $\mathcal{K}_\infty$ function $\rho$ such that the control system $\widehat\Sigma=(\R^{2n},\mathsf{D},\mathcal{D},\widehat{f})$\footnote{$\mathcal{D}$ is the set of all measurable, locally essentially bounded functions of time from intervals of the form \mbox{$]a,b[\subseteq\mathbb{R}$} to $\mathsf{D}$ with $a<0$ and $b>0$.} is U$_\exists$GAS with respect to the diagonal set $\Delta$, where: 
\begin{equation}
\widehat{f}(\zeta,\omega)=\left[
                \begin{array}{c}
                 f(\xi_1,\mathsf{sat}_\mathsf{U}(\omega_1+\rho(\mathbf{d}(\xi_1,\xi_2))\omega_2))\\ 
                 f(\xi_2,\mathsf{sat}_\mathsf{U}(\omega_1-\rho(\mathbf{d}(\xi_1,\xi_2))\omega_2))
                \end{array}
                \right],
\end{equation}
$\zeta=\left[\xi_1^T,\xi_2^T\right]^T$, $\mathsf{D}=\mathsf{U}\times\mathcal{B}_1(0)$, and $\omega=\left[\omega_1^T,\omega_2^T\right]^T$.
\end{lemma} 

\begin{proof}
The proof was inspired by the proof of Proposition 5.3 in \cite{angeli}. Since $\Sigma$ is $\delta_\exists$-ISS, equipped with the metric $\mathbf{d}$, there exists some $\mathcal{KL}$ function $\beta$ and $\mathcal{K}_\infty$ function $\gamma$ such that:
\begin{equation}\label{delta-ISS}
\mathbf{d}(\xi_{x\upsilon}(t),\xi_{x'\upsilon'}(t))\leq\max\{\beta(\mathbf{d}(x,x'),t),\gamma(\Vert\upsilon-\upsilon'\Vert_\infty)\}.
\end{equation} 
Note that inequality (\ref{delta-ISS}) is a straightforward consequence of inequality (\ref{delta_ISS}) (see Remark 2.5 in \cite{sontag3}).
Using Lemma \ref{lemma1} and the proposed metric $\widehat{\mathbf{d}}$ in (\ref{metric}), we have: $\mathbf{d}(x,x')=\widehat{\mathbf{d}}(z,\Delta)$, where $z=\left[x^T,x'^T\right]^T$. Without loss of generality we can assume $\alpha(r)=\beta(r,0)>r$ for any $r\in\R^+$. Let $\rho$ be a $\mathcal{K}_\infty$ function satisfying $\rho(r)\leq\frac{1}{2}\gamma^{-1}\circ\left(\alpha^{-1}(r)/4\right)$. Now we show that 
\begin{equation}\label{inequality0}
\gamma\left(\left\Vert2\omega_2(t)\rho\left(\widehat{\mathbf{d}}(\zeta_{z\omega}(t),\Delta)\right)\right\Vert\right)\leq\widehat{\mathbf{d}}(z,\Delta)/2,
\end{equation}
for any $t\in\R_0^+$, any $z\in\R^{2n}$, and any $\omega\in\mathcal{D}$. Since $\gamma$ is a $\mathcal{K}_\infty$ function and $\omega_2(t)\in\mathcal{B}_1(0)$, it is enough to show 
\begin{equation}\label{inequality1}
\gamma\left(2\rho\left(\widehat{\mathbf{d}}(\zeta_{z\omega}(t),\Delta)\right)\right)\leq\widehat{\mathbf{d}}(z,\Delta)/2.
\end{equation}
Since 
\begin{equation}
\gamma\left(2\rho\left(\widehat{\mathbf{d}}(\zeta_{z\omega}(0),\Delta)\right)\right)=\gamma\left(2\rho\left(\widehat{\mathbf{d}}(z,\Delta)\right)\right)\leq\alpha^{-1}\left(\widehat{\mathbf{d}}(z,\Delta)\right)/4\leq\widehat{\mathbf{d}}(z,\Delta)/4,
\end{equation}
and $\varphi(z)=\widehat{\mathbf{d}}(z,\Delta)$ is a continuous function (see proof of Theorem \ref{theorem1}), then for all $t\in\R_0^+$ small enough, we have $\gamma\left(2\rho\left(\widehat{\mathbf{d}}(\zeta_{z\omega}(t),\Delta)\right)\right)\leq\widehat{\mathbf{d}}(z,\Delta)/4$. Now, let 
\begin{equation}
t_1=\inf\left\{t>0\,\,|\,\,\gamma\left(2\rho\left(\widehat{\mathbf{d}}(\zeta_{z\omega}(t),\Delta)\right)\right)>\widehat{\mathbf{d}}(z,\Delta)/2\right\}.
\end{equation}
Clearly $t_1>0$. We will show that $t_1=\infty$. Now, assume by contradiction that $t_1<\infty$. Therefore, the inequality (\ref{inequality1}) holds for all $t\in[0,t_1)$. Hence, for almost all $t\in[0,t_1)$, one obtains:
\begin{equation}\label{inequality2}
\gamma\left(\left\Vert2\omega_2(t)\rho\left(\widehat{\mathbf{d}}(\zeta_{z\omega}(t),\Delta)\right)\right\Vert\right)\leq\gamma\left(2\rho\left(\widehat{\mathbf{d}}(\zeta_{z\omega}(t),\Delta)\right)\right)\leq\widehat{\mathbf{d}}(z,\Delta)/2<\alpha\left(\widehat{\mathbf{d}}(z,\Delta)\right)/2.
\end{equation} 
Let $\upsilon$ and $\upsilon'$ be defined as:
\begin{eqnarray}
\nonumber
\upsilon(t)&=&\mathsf{sat}_\mathsf{U}\left(\omega_1(t)+\rho\left(\widehat{\mathbf{d}}(\zeta_{z\omega}(t),\Delta)\right)\omega_2(t)\right),\\\notag
\upsilon'(t)&=&\mathsf{sat}_\mathsf{U}\left(\omega_1(t)-\rho\left(\widehat{\mathbf{d}}(\zeta_{z\omega}(t),\Delta)\right)\omega_2(t)\right).
\end{eqnarray}
By using (\ref{convex}), we obtain: $\Vert \upsilon(t)-\upsilon'(t)\Vert\leq\left\Vert2\omega_2(t)\rho\left(\widehat{\mathbf{d}}(\zeta_{z\omega}(t),\Delta)\right)\right\Vert$. Using (\ref{delta-ISS}) and (\ref{inequality2}), we have: 
\begin{equation}
\widehat{\mathbf{d}}(\zeta_{z\omega}(t),\Delta)=\mathbf{d}\left(\xi_{x\upsilon}(t),\xi_{x'\upsilon'}(t)\right)\leq\beta\left(\mathbf{d}(x,x'),0\right)=\beta\left(\widehat{\mathbf{d}}(z,\Delta),0\right)=\alpha\left(\widehat{\mathbf{d}}(z,\Delta)\right), 
\end{equation}
for any $t\in[0,t_1]$ and any $z=\left[x^T,x'^T\right]^T\in\R^{2n}$ which implies that $\gamma\left(2\rho\left(\widehat{\mathbf{d}}(\zeta_{z\omega}(t),\Delta)\right)\right)\leq\widehat{\mathbf{d}}(z,\Delta)/4$, contradicting the definition of $t_1$. Therefore, $t_1=\infty$ and inequality (\ref{inequality0}) is proved for all $t\in\R_0^+$. Therefore, using (\ref{delta-ISS}) and (\ref{inequality0}), we obtain:
\begin{eqnarray}
\widehat{\mathbf{d}}(\zeta_{z\omega}(t),\Delta)=\mathbf{d}\left(\xi_{x\upsilon}(t),\xi_{x'\upsilon'}(t)\right)&\leq&\max\left\{\beta\left(\mathbf{d}(x,x'),t\right),\gamma\left(\Vert\upsilon-\upsilon'\Vert_\infty\right)\right\}\\\notag&\leq&\max\left\{\beta\left(\mathbf{d}(x,x'),t\right),\gamma\left(\left\Vert2\omega_2\rho\left(\widehat{\mathbf{d}}(\zeta_{z\omega},\Delta)\right)\right\Vert_\infty\right)\right\}\\\notag&\leq&\max\left\{\beta\left(\widehat{\mathbf{d}}(z,\Delta),t\right),\widehat{\mathbf{d}}(z,\Delta)/2\right\},
\end{eqnarray}
for any $z=\left[x^T,x'^T\right]^T\in\R^{2n}$, any $\omega\in\mathcal{D}$, and any $t\in\R_0^+$. Since $\beta$ is a $\mathcal{KL}$ function, it can be readily seen that for each $r>0$ if $\widehat{\mathbf{d}}(z,\Delta)\leq r$, then there exists some $T_r\geq0$ such that for any $t\geq T_r$, $\beta\left(\widehat{\mathbf{d}}(z,\Delta),t\right)\leq r/2$ and, hence, $\widehat{\mathbf{d}}(\zeta_{z\omega}(t),\Delta)\leq r/2$. For any $\varepsilon\in\R^+$, let $k$ be a positive integer such that $2^{-k}r<\varepsilon$. Let $r_1=r$ and $r_i=r_{i-1}/2$ for $i\geq2$, and let $\tau=T_{r1}+T_{r2}+\cdots+T_{rk}$. Then, for $t\geq\tau$, we have $\widehat{\mathbf{d}}(\zeta_{z\omega}(t),\Delta)\leq2^{-k}r<\varepsilon$ for all $\widehat{\mathbf{d}}(z,\Delta)\leq r$, all $\omega\in\mathcal{D}$, and all $t\geq\tau$. Therefore, it can be concluded that the set $\Delta$ is a uniform global attractor for the control system $\widehat\Sigma$. Furthermore, since $\widehat{\mathbf{d}}(\zeta_{z\omega}(t),\Delta)\leq\beta\left(\widehat{\mathbf{d}}(z,\Delta),0\right)$ for all $t\in\R_0^+$, all $z\in\R^{2n}$, and all $\omega\in\mathcal{D}$, the control system $\widehat\Sigma$ is uniformly globally stable and as showed in \cite{teel}, it is U$_\exists$GAS.
\end{proof}

The next theorem provide characterization of $\delta_\exists$-ISS in terms of existence of a $\delta_\exists$-ISS Lyapunov function.
\begin{theorem}
Consider a control system $\Sigma$. If $\mathsf{U}$ is compact and convex and $\mathbf{d}$ is a metric such that the function $\psi(x)=\mathbf{d}(x,y)$ is continuous\footnote{Here, continuity is understood with respect to the Euclidean metic.} for any $y\in\R^n$ then the following statements are equivalent: 
\begin{itemize}
\item[(1)] $\Sigma$ is forward complete and there exists a $\delta_\exists$-ISS Lyapunov function, equipped with metric ${\mathbf{d}}$.
\item[(2)] $\Sigma$ is $\delta_\exists$-ISS, equipped with metric $\mathbf{d}$.
\end{itemize}
\end{theorem}
\begin{proof}
The proof from (1) to (2) has been showed in Theorem 2.6 in \cite{majid4}, even in the absence of the compactness and convexity assumptions on $\mathsf{U}$ and the continuity assumption on $\mathbf{d}$. We now prove that (2) implies (1). As we proved in Lemma \ref{lemma2}, since $\Sigma$ is $\delta_\exists$-ISS, it implies that the control system $\widehat\Sigma$, defined in Lemma \ref{lemma2}, is U$_\exists$GAS. Since $\psi(x)=\mathbf{d}(x,y)$ is continuous for any $y\in\R^n$, it can be easily verified that $\widehat\psi(z)=\widehat{\mathbf{d}}(z,z')$ is continuous for any $z'\in\R^{2n}$, where the metric $\widehat{\mathbf{d}}$ was defined in Lemma \ref{lemma1}. Using Theorem \ref{theorem1}, we conclude that there exists a U$_\exists$GAS Lyapunov function V, with respect to $\Delta$, for $\widehat\Sigma$. Thanks to the special form of $\widehat\Sigma$ and using the equality (\ref{equality}), the function $V$ satisfies: 
\begin{equation}
\underline\alpha(\mathbf{d}(x,x'))\leq V(x,x')\leq\overline\alpha(\mathbf{d}(x,x')),
\end{equation}
for some $\mathcal{K}_\infty$ functions $\underline\alpha,\overline\alpha$, any $x,x'\in\R^n$, and 
\begin{equation}\label{deriv}
\frac{\partial{V}}{\partial{x}}f(x,\mathsf{sat}_\mathsf{U}(d_1+\rho(\mathbf{d}(x,x')))d_2)+\frac{\partial{V}}{\partial{x'}}f(x',\mathsf{sat}_\mathsf{U}(d_1-\rho(\mathbf{d}(x,x'))d_2))\leq-\kappa V(x,x'),
\end{equation}
for some $\kappa\in\R^+$ and any $\left[d_1^T,d_2^T\right]^T\in\mathsf{D}$. By choosing $d_1=(u+u')/2$ and $d_2=(u-u')/(2\rho(\mathbf{d}(x,x')))$ for any $u,u'\in\mathsf{U}$, it can be readily checked that $\left[d_1^T,d_2^T\right]^T\in\mathsf{U}\times\mathcal{B}_1(0)$, whenever $2\rho(\mathbf{d}(x,x'))\geq\Vert u-u'\Vert$. Hence, using (\ref{deriv}), we have :
\begin{equation}\label{cond2}
\varphi(\mathbf{d}(x,x'))\geq\Vert u-u'\Vert\Rightarrow\frac{\partial{V}}{\partial{x}}f(x,u)+\frac{\partial{V}}{\partial{x'}}f(x',u')\leq-\kappa V(x,x'),
\end{equation}
where $\varphi(r)=2\rho(r)$. As showed in Remark 2.4 in \cite{sontag3}, there is no loss of generality in modifying inequality (\ref{cond2}) to
\begin{equation}
\frac{\partial{V}}{\partial{x}}f(x,u)+\frac{\partial{V}}{\partial{x'}}f(x',u')\leq-\widehat\kappa V(x,x')+\gamma(\Vert u-u'\Vert),
\end{equation}
for some $\mathcal{K}_\infty$ function $\gamma$ and some $\widehat\kappa\in\R^+$, which completes the proof.
\end{proof}


\bibliographystyle{alpha}
\bibliography{reference}
\end{document}